\documentclass[%review,
a4paper,10pt,pdftex]{amsart}

\usepackage[linkcolor=black,pdftex]{hyperref}

\usepackage{amssymb,amsmath,amsfonts,amsthm}%stmaryrd}
\usepackage[T1]{fontenc}
\usepackage{tensor}
\usepackage{textcomp}
\usepackage{epsfig}
\usepackage{graphicx}
\usepackage[enableskew]{youngtab}
\usepackage{tikz}
\usetikzlibrary{matrix}
\input xy 
\xyoption{all}
\xyoption{poly}
\usepackage{lscape}
\usepackage[hyperpageref]{backref} 
\usepackage{faktor}
\usepackage{ytableau}
\usepackage[nice]{nicefrac}
\usepackage{calc}
\theoremstyle{plain}

\newtheorem{satz}{Theorem}[section]
\newtheorem{theorem}{Theorem}[section]

\newtheorem{coro}{Corollary}[section]
\newtheorem{lemma}{Lemma}

\theoremstyle{definition}
\newtheorem{defi}{Definition}
\theoremstyle{definition}

\newtheorem*{rem}{Remark}

\theoremstyle{remark}

\DeclareMathOperator{\id}{id}

\DeclareMathOperator{\Hom}{\rm Hom}
\DeclareMathOperator{\End}{\rm End}
\DeclareMathOperator{\Aut }{\rm Aut}
\DeclareMathOperator{\Dim}{\rm \underline{dim}}

\DeclareMathOperator{\rk}{\rm rk}
\DeclareMathOperator{\inj}{\hookrightarrow}

\newcommand{\oTo}{\xymatrix{ \ar@{^{(}->}[r]|{\mathbf{O}}& }}
\newcommand{\cTo}{\xymatrix{ \ar@{^{(}->}[r]|{\mathbf{|}}& }}
\newcommand{\coTo}{\xymatrix{ \ar@{^{(}->}[r]|{\mathbf{O}}|{\mathbf{|}}& }}
\DeclareMathOperator{\surj}{\twoheadrightarrow }

\newcommand{\Mat }{\mathrm{Mat}}

\newcommand{\Gl}{\mathbf{Gl}}

\DeclareMathOperator{\soc}{soc}
\DeclareMathOperator{\rad}{rad}
\DeclareMathOperator{\Top}{top}

\newcommand{\si}{\sigma}

\newcommand{\la}{\lambda}
\newcommand{\La}{\Lambda}

\newcommand{\C}{\mathbb{C}}

\newcommand{\N}{\mathbb{N}}

\newcommand{\Q}{\mathbb{Q}}

\newcommand{\Z}{\mathbb{Z}}

\newcommand{\mcF}{\mathcal{F}}

\newcommand{\mcH}{\mathcal{H}}

\newcommand{\mcL}{\mathcal{L}}
\newcommand{\mcM}{\mathcal{M}}

\newcommand{\mcO}{\mathcal{O}}

\newcommand{\mcS}{\mathcal{S}}

\newcommand{\dd}{\underline{d}}

\newcommand{\ee}{\underline{e}}

\newcommand{\ddd}{\underline{\mathbf{d}}}
\newcommand{\eee}{\underline{\mathbf{e}}}

\DeclareMathOperator{\gdim}{gdim}

\newcommand{\fl}{\mathrm{Fl}}

\newcommand{\Rd}{\mathrm{R}(\dd)}

\newcommand{\FMd}{{\rm Fl}\binom{M}{\underline{\mathbf{d}}}}
\newcommand{\fYd}{{\rm f}\binom{Y}{\underline{\mathbf{d}}}}
\newcommand{\fYbd}{{\rm f}\binom{Y\setminus b}{\underline{\mathbf{d}}-\dd^1}}
\newcommand{\fYdk}{{\rm f}\left( \binom{Y}{\ddd} ; k\right) }
\newcommand{\fldk}{{\rm f}\left( \binom{\lambda }{\ddd } ; k\right) }
\newcommand{\fldek}{{\rm f}\left( \binom{Y}{\ddd } ; e_{\ddd} - k\right) }
\newcommand{\fYbdm}{{\rm f}\left( \binom{Y\setminus{b_m}}{\underline{\mathbf{d}}-\dd^1} ; k+m-s\right) }
\newcommand{\GrMe}{{\rm Gr}\binom{M}{\ee}}

\newcommand{\FMdone}{{\rm Fl}\binom{M_1}{{\underline{\mathbf{d}}}^1}}
\newcommand{\FMdt}{{\rm Fl}\binom{M_t}{{\underline{\mathbf{d}}^t}}}

\newcommand{\FNe}{{\rm Fl}\binom{N}{\underline{\mathbf{e}}}}

\newcommand{\rp}{\rm rp}

\newcommand{\vierzehn}{14}
\newcommand{\dreizehn}{13}
\newcommand{\zwoelf}{12}
\newcommand{\elf}{11}
\newcommand{\zehn}{10}

\begin{document}

%\begin{frontmatter}

\title[Cell decomposition of quiver flag varieties]{Cell decompositions of quiver flag varieties for nilpotent representations of the oriented cycle}

\author{Julia Sauter}
\address{Faculty of Mathematics, 
Bielefeld University
Postfach 100 131
D-33501 Bielefeld}

\begin{abstract}
Generalizing Schubert cells in type $A$ and a cell decomposition of Springer fibres in type $A$ found by L. Fresse we prove that varieties of complete flags in nilpotent representations of an oriented cycle admit an affine cell decomposition parametrized by multi-tableaux. 
We show that they carry a torus operation and describe the $T$-equivariant cohomology using Goresky-Kottwitz-MacPherson-theory.  
As an application of the cell decomposition we obtain a vector space basis of certain modules (for quiver Hecke algebras of nilpotent representations of this quiver), similar modules have been studied by Kato as analogues of standard modules. 
\end{abstract}
%
%\begin{keyword}
%\texttt{elsarticle.cls}\sep \LaTeX\sep Elsevier \sep template
%\MSC[2010] 00-01\sep  99-00
%\end{keyword}
%
%\end{frontmatter}

\subjclass[2010]{Primary 14M15; Secondary 16G20}

%16  (1959-now) Associative rings and algebras [For the commutative case, see 13-XX]
%16G  View Publications (1991-now) Representation theory of rings and algebras
%16G60  View Publications (1991-now) Representation type (finite, tame, wild, etc.)
%14  View Publications (1940-now) Algebraic geometry
%14L  View Publications (1973-now) Algebraic groups [For linear algebraic groups, see 20Gxx; for Lie algebras, see 17B45]
%14L30  View Publications (1980-now) Group actions on varieties or schemes (quotients) [See also 13A50, 14L24, 14M17]
%14  View Publications (1940-now) Algebraic geometry
%14M  View Publications (1973-now) Special varieties
%14M15  View Publications (1973-now) Grassmannians, Schubert varieties, flag manifolds [See also 32M10, 51M35]

\keywords{Quiver flag variety, Schubert cell, quiver Hecke algebra}

\maketitle
 
\section{Introduction}  
We study varieties of complete flags in nilpotent representations associated to an oriented cycle. In this situation the filtrations by radicals and socles play a special role (for an indecomposable module they are the functorial filtrations for a string module). Restriction to the one-dimensional subspace and relative position with respect to these flags gives us (recursively) a cell decomposition into affine cells. This method goes back to a work of Spaltenstein \cite{Sp}, compare section 3.  
For the loop (or Jordan) quiver with the zero endomorphism we get the Schubert cell decomposition for $\Gl_n$ and for more general representations we reobtain the cell decomposition for Springer fibres in type $A$ studied by L. Fresse  in \cite{Fr}.  
An efficient way to parametrize the cells is book-keeping of the relative position in terms of multi-tableaux (see section 4) and this gives a combinatorial tool to describe the Betti numbers of these quiver flag varieties. 

We also observe that the torus action which is given by scalar multiplication on each indecomposable summand operates on the quiver flag variety with finitely many fixed points and one-dimensional orbits. In fact, every cell is a limit cell of a unique torus fix point in it and the cell decomposition is an instance of a Bialynicki-Birula decomposition (of a non-smooth variety, see e.g. \cite{Car}). The theory developped by Goresky, Kottwitz and MacPherson (see \cite{GKM}) gives a description of the torus equivariant cohomology of these varieties. 

As an application, the Borel-Moore homology groups appear as modules for quiver Hecke algebras of nilpotent representations of the oriented cycle, 
see \cite{Ka3} for the definition in a more restrictive situation and a general study of their properties (Kato's situation does not apply because it requires all multiplicites $L_{\la}$ to be non-zero).
Quiver Hecke algebras are known to be graded Morita equivalent to (positively graded) standardly stratified algebras in the sense of Mazorchuk \cite{Ma} and under this equivalence Kato's standard modules correspond to (Mazorchuk's) two types of standard modules.

\section{Nilpotent representations of the oriented cycle and quiver flag varieties}

Let $K$ be an algebraically closed field. Let $A=KQ$ where $Q$ is the oriented cycle with vertices $Q_0=\{1,\ldots , n\}$ identified with their residue classes in $\Z/n\Z$ and  $Q_1=\{a_i\colon i\to i+1 \mid i\in \Z/n\Z\}$. An $A$-module $M$ is given by vector spaces $V_1, \ldots , V_n$ and linear maps $f_i\colon V_i\to V_{i+1}$. Its $Q_0$-graded dimension is given by $\Dim M=(\dim V_1, \ldots , \dim V_n)$.  We will only look at finite-dimensional nilpotent representations, that means $\dim V_i<\infty $ and 
$f_{i+n-1} \circ \cdots \circ f_{i+1}\circ f_i\colon V_i\to V_i$ is nilpotent for every $i\in Q_0$.  

\paragraph{Indecomposable nilpotent $A$-modules} For $1\leq i\leq n$ we write $S_i=E_i[1]$ for the simple left $A$-module supported in the vertex $i$. For $j\in Q_0, \ell\in \N $ we write $E_{j}[\ell]$ for the unique indecomposable $A$-module with socle $S_j$ and $K$-vector space dimension $\ell $.
We have  $\soc E_{j}[\ell ] = S_j\subset E_{j}[\ell ]$ the inculsion of the socle, $\rad E_{j}[\ell ]= E_{j}[\ell-1] \subset E_{j}[\ell ]$ (with $E_j[0]:=0$) the inclusion of the radical and $E_j[\ell] \surj \Top E_j[\ell] = S_{j-\ell +1}$ the quotient map to the top.

\begin{defi}
Let $M$ be a nilpotent $A$-module of dimension $\Dim M =:\dd \in \N_0^{Q_0}$ and 
$\ddd := (0=\dd^0 , \dd^1 ,\ldots , \dd^{r}=\dd)$ with $\dd^k \in \N_0^{Q_0}$ be a sequence  We define 
\[
\FMd := \{ U=(U^0\subset \cdots \subset U^{r}=M) \mid U^k \; A 
\text{-submodule}, \; \Dim U^k =\dd^k \}
\]

This defines a projective $K$-variety, we call it the \textbf{quiver flag variety} for $(M, \ddd )$. We will always assume that the flags are complete i.e. $\left| \dd^{t+1}\right|- \left|\dd^t \right| =1 , \; 1\leq t \leq r -1$ (where $\left| v\right| =\sum_{i=1}^nv_i$ for $v\in \N_0^{Q_0}$), with one exception: If $\ddd=(0, \ee , \dd)$ we denote the quiver flag variety by $\GrMe$. 
\end{defi}

For $d\in \N$ we denote $\fl (d)$ the variety of complete flags in $K^d$ and $\fl(0):=pt$.   
\paragraph{Relative position}
Let $\fl (\ddd ):=\prod_{i\in Q_0} \fl (\dd_i), \fl (\eee ) :=\prod_{i\in Q_0} \fl (\ee_i )$ with $\dd_i=(0=d^0_i \leq d_i^1 \leq \cdots \leq d_i^{r}),  \ee_i=(0=e^0_i \leq e_i^1 \leq \cdots \leq e_i^{\mu})$, $d_i^{r}=e_i^{\mu}$ for all $i\in Q_0$. The relative position map  
$\rp \colon \fl(\ddd )\times \fl (\eee ) \to \prod_{i\in Q_0} \Mat \bigl((r +1)\times (\mu +1) , \N_0 \bigr)$ is defined as  
$
\rp (U_i^{\bullet}, W_j^{\bullet})_{i,j\in Q_0} := \bigl((\dim U_i^k \cap W_i^\ell)_{k,\ell}\bigr)_{i\in Q_0}. 
$ 
Now fix $W\in \fl (\eee )$ and given $w\in \prod_{i\in Q_0} \Mat \bigl((r +1)\times (\mu +1) , \N_0 \bigr)$ we set 
\[ \fl (\ddd )_{w} :=\{ U\in \fl (\ddd )\mid \rp (U,W)=w\}. \]
For our fixed representation $M$ we will assume from now on $V_i=K^{d_i}$, $i\in Q_0$ and for the quiver flag variety from before, we set 
\[ 
\FMd_{w} := \FMd \cap \fl (\ddd )_{w} 
\]
This gives the \textbf{stratification by relative position (with the fixed flag $W$) }in $\FMd$.

\section{Spaltenstein's fibration and the cell decomposition}

Let $M$ be an $A$-module. %We write $\underline{M}$ for the underlying $Q_0$-graded vector space. 
If $L$ is $1$-dimensional $Q_0$-graded subvector space of $M$, the inclusion 
$j\colon L\inj M$ is $A$-linear if and only if $L\subset \soc (M)$. Thus for a dimension vector $\ee$ with $\ee =e_i$ for some $i\in Q_0$,  we have an isomorphism ${\rm Gr} \binom{M}{\ee} \cong {\rm Gr} \binom{s_i  }{1 }=\mathbb{P}^{s_i -1}$ where $\underline{s}:= \Dim \soc (M)$. 

We need the following preparation. 
Let $i\in Q_0$ and denote by $M_{(i)}$ the maximal subrepresentation of a representation $M$ such that $\soc (M_{(i)})$ is a direct sum of copies of $S_i$. We get $M=\bigoplus_{i\in Q_0}M_{(i)}$ and we can see $\Aut (M_{(i)})\subset \Aut (M)$ as a subgroup.  
We denote by $F=F_M$ the underlying $Q_0$-graded flag of the following flag of submodules
$ 0\subset \soc (M) \cap \rad^m (M) \subset \cdots \subset \soc (M)\cap \rad^2 (M) \subset  \soc (M) \cap \rad (M) \subset \soc (M) $. 

Now fix $i\in Q_0$ and let $s_i:= \dim \soc(M_{(i)})$. We can consider $F_i$ as a flag in the vector space $\soc M_{(i)}$. We denote by  $P=P_i \subset \Gl_{s_i}$ the stabilizer of this flag.  
The restriction to the socle gives a group homomorphism 
\[ 
\varphi\colon \Aut(M_{(i)} ) \to \Gl_{s_i}.
\]
We fix a vector space basis for $M_{(i)}$ which is compatible with the Krull-Schmidt decomposition and a complete flag refining $F_i$ such that its stabilizer $B\subset P$ is a lower triangular standard Borel.

\begin{lemma} The image of $\varphi $ is $P$ and 
there is a group morphism $\theta \colon P\to \Aut (M_{(i)}) $ such that $\varphi\circ \theta =\id_P$. 
\end{lemma}

\begin{proof} Clearly the image is contained in $P$ since any $A$-linear map maps radicals to radicals. We look at the $K$-algebra homomorphism $\Phi$ given by taking socle 
whose restriction to the units gives the map   
$\varphi$. 
\[ \Phi \colon \End_A (M_{(i)}) \to \mathfrak{p}:=\{ f\in \End_K (\soc M_{(i)})\mid f (F_i^k)\subset F_i^k, k\leq m\} .\]
If we fix a vector space basis for $M$ compatible with the direct sum decomposition, we can define a $K$-algebra homomorphism $\Theta \colon \mathfrak{p}\to \End_A(M_{(i)})$ such that 
$\Phi\circ \Theta =\id_{\mathfrak{p}}$ as follows. 
Any coordinate $(s,t)$ in $\End_K(\soc M_{(i)})$ corresponds (by our choice of a  basis) to a pair of socles of direct summands $E_i[\ell]$ and $E_i[k]$ of $M_{(i)}$. If $k<\ell$, then there is no nonzero $A$-linear map $E_i[\ell]\to E_i[k]$, we set $\theta_{s,t} =0$. If $k\geq \ell $ we fix the (unique) inclusion $\theta_{s,t}\colon E_i[\ell] \subset E_i[k]$ of the submodule. Then, $\Theta ((a_{s,t})) := \sum_{s,t} a_{s,t}\theta_{s,t}$ defines the desired map. 
\end{proof}
We call two points $U^\bullet, W^\bullet \in \FMd$ equivalent if 
the quotients of $M/U^k$ and $M/V^k$ are isomorphic $A$-modules for $1\leq k\leq r$. There are only finitely many equivalence classes, we denote by $sp \colon \FMd \to \mcS, U^\bullet \mapsto ([M/U^k])_{1\leq k\leq r}$ the map to the equivalence class.  

Now we can prove the following analogue of a result of Spaltenstein \cite[Lemma, p. 453]{Sp}.

\begin{theorem}[Spaltenstein's fibration] Let $Q$ be the oriented cycle with $n$ vertices, and $\ddd=(\dd^0,\dd^1,\ldots , \dd^{r}=\dd)$ a complete dimension filtration. Let $M$ be a $\dd$-dimensional nilpotent $A$-module, for $w\in \prod_{i\in Q_0} \Mat (1 \times (s_i +1), \N_0)$ we write $()_w$ for the relative position with respect to a complete flag refining $\soc(M) \cap \rad^{\bullet}(M)$. 
Then, there is an isomorphism of algebraic varieties 
\[ 
 f\colon p^{-1} ({\rm Gr} \binom{M}{\dd^{1}}_{w}) \to
{\rm Gr} \binom{M}{\dd^{1}}_{w} \times \FNe 
\]  
with $\eee:= (\dd^1-\dd^1,\dd^2-\dd^1,\ldots , \dd^{r }-\dd^1)$, $p\colon \FMd \to {{\rm Gr} \binom{M}{\dd^{1}}}$ is the map forgetting all but the first subspace, and $N = M/U_0$ with $(U_0\subset M) \in  {{\rm Gr} \binom{M}{\dd^{1}}}_w $ arbitrary, such that the following 
diagram is commutative 
\[ 
\xymatrix{ &{\rm Gr} \binom{M}{\dd^{1}}_{w} & \\ 
p^{-1} ({\rm Gr} \binom{M}{\dd^{1}}_{w})\ar[ur]^{p} \ar[rr]^f \ar[dr]_{sp} &&{\rm Gr} \binom{M}{\dd^{1}}_{w} \times \FNe  \ar[ul]_{pr_1}\ar[dl]^{([N], sp)\circ pr_2} \\ &\mcS .&
}
\]
\end{theorem}
  
\begin{proof}
Let $L\in {\rm Gr} \binom{M}{\dd^{1}}_w=\{[0:\cdots :0:1:x_1:\ldots :x_s]\in \mathbb{P}^{s_j -1}\mid x_i \in K, 1\leq i\leq s\}$. Observe, that by definition $B$ operates transitive on ${\rm Gr} \binom{M}{\dd^{1}}_w$ since it is a $B$-orbit. 
Using the previous lemma we can find an algebraic map $\phi \colon  
{\rm Gr} \binom{M}{\dd^{1}}_w \to \Aut (M), \; L\mapsto \phi_L$ with $\phi_L (L)=U_0$. More precisely, if $L$ corresponds to the column $(0,\ldots ,0, 1, x_1, \ldots , x_s)^t$ and $U_0$ to $(0,\ldots ,0 ,1, 0,\ldots ,0)^t$, then we 
take the image of 
\[
\begin{tikzpicture}[baseline=(current bounding box.center)]
\matrix (m) [matrix of math nodes,nodes in empty cells,right delimiter={]},left delimiter={[} ]{
1  &    &      &  &  &   \\
   &    &      &  &  &   \\
   &    &      &  &  &   \\
   &    & 1    &  &  &   \\
   &    & -x_1 &  &  &   \\
   &    &      &  &  &   \\
   &    & -x_s &  &  & 1 \\ 
} ;
\draw[dotted] (m-1-1)-- (m-4-3);
\draw[dotted] (m-4-3)-- (m-7-6);
\draw[dotted] (m-5-3)-- (m-7-3);
\end{tikzpicture}
\]
under $B\subset P\subset \Aut(M_{(i)} \subset \Aut (M)$. 
We define 
\[
\begin{aligned}
f\colon p^{-1} ({\rm Gr} \binom{M}{\dd^{1}})_{w} &\to {\rm Gr} \binom{M}{\dd^{1}}_{w} \times \FNe \\
U=(U^1\subset \cdots \subset U^{r}=M) &\mapsto (U^1, \phi_{U^1}(U)/U_0 ) .
\end{aligned}
\]
This is a morphism of algebraic varieties. To find the inverse, we consider $\pi\colon M\to M/U_0 =N$ the canonical projection and define 
\[
\begin{aligned}
{\rm Gr} &\binom{M}{\dd^{1}}_{w} \times \FNe \to p^{-1} ({\rm Gr} \binom{M}{\dd^{1}}_{w})\\
&\left( L, V=(V^1\subset \cdots \subset V^{r -1}=N)\right) \\
&\quad \mapsto (L \subset \phi_L^{-1}\pi^{-1}(V^1)\subset \phi_L^{-1}\pi^{-1}V^2 \subset \cdots 
\subset \phi_L^{-1}\pi^{-1}V^{r -1 }=M).
\end{aligned}
\]
\end{proof}

\begin{defi} Let $X$ be a scheme. An \text{affine cell decomposition} is a filtration 
\[X=X_m \supset X_{m-1}\supset \cdots \supset X_0\supset X_{-1}=\emptyset \]
by closed subschemes, with each $X_{i}\setminus X_{i-1}$ is a disjoint union of finitely many schemes $U_{ij}$ isomorphic to affine spaces $\mathbb{A}^{n_{ij}}$. 
\end{defi}

\begin{coro} 
Let $\FMd$ be a complete quiver flag variety for the oriented circle and $M$ be a nilpotent representation. Then, it has an affine cell decomposition. If $K=\C$ or $\overline{\Q_{\ell}}$ then it is pure.   
\end{coro}

Note that there is a closed embedding by forgetting the $A$-module structure 
$\kappa \colon \FMd \to \fl (\ddd ) =:\mcF$. Forgetting all other then the first subspace gives a commutative square 
\[ 
\xymatrix{ \mcF \ar[r]^q & \mathbb{P}^{d_i -1} \\
\FMd \ar[r]^p\ar[u]^{\kappa} & \mathbb{P}^{s_i-1}\ar[u]^{\rho} 
}
\] 
where the vertical maps are closed immersions. By choosing appropiate cells in $\mcF$, there is an affine cell decomposition in $\mcF$ such that 
the intersection with $\FMd$ is a union of cells in $\mcF$. This implies that the open complement of $\FMd$ in $\mcF$ also has an affine cell decomposition, this is the main ingredient to prove the following theorem, for $n=1$ compare \cite[section 4.4, 4.5]{DP}.

\begin{satz}
Let be $K=\C$. Let $Q$ be the oriented cycle, $M$ be a nilpotent representation and $\dd$ a complete dimension filtration. 
The pullback along $\kappa \colon \FMd \to \mcF$ induces a surjective ring homomorphism 
on singular cohomology 
\[ 
\kappa^* \colon H^*(\mcF) \to H^*(\FMd).  
\]
\end{satz}

\paragraph{proof:} 
By the universal coefficient theorem for projective varieties, it suffices to show that $\kappa_* \colon H_*^{BM}(\FMd) \to H_*^{BM}(\mcF )$ is injective. Let $U$ be the open complement of $\FMd$ in $\mcF$. By the construction of the Spaltenstein fibration, we get that $\mcF$ has a cell decomposition continuing the one of $\FMd$, therefore $U$ also has also a cell decomposition. Since we have odd-degree vanishing also for $U$ the localization sequence gives a short exact sequence 
\[
0\to H_*^{BM} (\FMd ) \to H_*^{BM} (\mcF ) \to H_*^{BM} (U) \to 0 
\]
\hfill $\Box $
% Later ?????????????????????
%We remark that if $X$ is a projective complex variety with an affine cell 
%decomposition and let $H^*(X,\Q )$ denote singular cohomology (with rational coefficients), then one has 
%$H^{\ell}(X,\Q ) =0$ for $\ell$ odd and $H^{2m} (X, \Q )$ has a basis given by the cycles of the closures of the $m$-dimensional cells. 

%------------------------------------------------------------------------
\section{Parametrizing cells by multi-tableaux} \label{PS}

To the nilpotent indecomposable $A$-module $E_i[\ell]$ we associate a row of $\ell $ boxes, indexing the columns by the elements $i-\ell +1,i-\ell, \ldots , i$  in $\Z/n$ from left to right. 
To any nilpotent $A$-module we associate a multipartition $Y=(Y_j)_{j\in \Z/n}$ by taking the young diagram $Y_j$ corresponding to all indecomposable summands which have top supported on $j\in \Z/n$. For example for $n=3$ and \[M=(E_3[3]^2\oplus E_2[2]) \oplus (E_2[4] \oplus E_3[2]) \] we visualize the multipartition $Y_M= ((3,3,2), (4,2), \emptyset )$ as follows, see the figure on the left.  
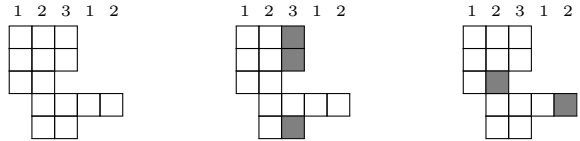
\[ 
\begin{array}{c}
\makebox[9pt][c]{\text{\tiny{1}}}\makebox[9pt][c]{\text{\tiny{2}}}\makebox[9pt][c]{\text{\tiny{3}}}\makebox[9pt][c]{\text{\tiny{1}}}\makebox[9pt][c]{\text{\tiny{2}}}\\
\ytableausetup{smalltableaux}
\begin{ytableau}
{} & & & \none & \none  \\
& & & \none & \none \\
& &\none &\none & \none  \\
\none & & & & \\
\none & & & \none & \none 
\end{ytableau}
\end{array}
\quad \quad \quad  
\begin{array}{c}
\makebox[9pt][c]{\text{\tiny{1}}}\makebox[9pt][c]{\text{\tiny{2}}}\makebox[9pt][c]{\text{\tiny{3}}}\makebox[9pt][c]{\text{\tiny{1}}}\makebox[9pt][c]{\text{\tiny{2}}}\\
\ytableausetup{smalltableaux}
\begin{ytableau}
{} & & *(gray) &  \none & \none  \\
& & *(gray) & \none & \none \\
& &\none &\none & \none  \\
\none & & & & \\
\none & & *(gray) & \none & \none 
\end{ytableau}
\end{array}
\quad \quad \quad 
\begin{array}{c}
\makebox[9pt][c]{\text{\tiny{1}}}\makebox[9pt][c]{\text{\tiny{2}}}\makebox[9pt][c]{\text{\tiny{3}}}\makebox[9pt][c]{\text{\tiny{1}}}\makebox[9pt][c]{\text{\tiny{2}}}\\
\ytableausetup{smalltableaux}
\begin{ytableau}
{} & & & \none & \none  \\
& & & \none & \none \\
& *(gray) &\none &\none & \none  \\
\none & & & & *(gray) \\
\none & & & \none & \none 
\end{ytableau}
\end{array}
\]
In the middle and the right hand side figure we shaded the socle at $3$ and the socle at $2$ respectively, the socle at $1$ is zero. 
From now on we choose a basis of $M$ such that each box corresponds to a basis vector and we order the basis vectors starting at the first row from left to right and then the second row from left to right and so on.  
When we include a line into the vector space given by the socle at $3$, we find a first 
direct summand (with respect to the fixed basis from before) where the inclusion is nonzero, we call the corresponding box the \textbf{pivot box}. Looking at lines with the same pivot box defines a Schubert cell, 
e.g. we indicate the Schubert cells as follows, the pivot box gets a $1$, the number of stars indicates the dimension of the cell
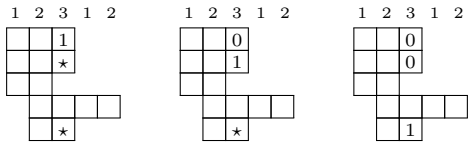
\[ 
\begin{array}{c}
\makebox[9pt][c]{\text{\tiny{1}}}\makebox[9pt][c]{\text{\tiny{2}}}\makebox[9pt][c]{\text{\tiny{3}}}\makebox[9pt][c]{\text{\tiny{1}}}\makebox[9pt][c]{\text{\tiny{2}}}\\
\ytableausetup{smalltableaux}
\begin{ytableau}
{} & & 1 &  \none & \none  \\
& & \star & \none & \none \\
& &\none &\none & \none  \\
\none & & & & \\
\none & & \star & \none & \none 
\end{ytableau}
\end{array}
\quad 
\begin{array}{c}
\makebox[9pt][c]{\text{\tiny{1}}}\makebox[9pt][c]{\text{\tiny{2}}}\makebox[9pt][c]{\text{\tiny{3}}}\makebox[9pt][c]{\text{\tiny{1}}}\makebox[9pt][c]{\text{\tiny{2}}}\\
\ytableausetup{smalltableaux}
\begin{ytableau}
{} & & 0 &  \none & \none  \\
& & 1 & \none & \none \\
& &\none &\none & \none  \\
\none & & & & \\
\none & & \star & \none & \none 
\end{ytableau}
\end{array}
\quad 
\begin{array}{c}
\makebox[9pt][c]{\text{\tiny{1}}}\makebox[9pt][c]{\text{\tiny{2}}}\makebox[9pt][c]{\text{\tiny{3}}}\makebox[9pt][c]{\text{\tiny{1}}}\makebox[9pt][c]{\text{\tiny{2}}}\\
\ytableausetup{smalltableaux}
\begin{ytableau}
{} & & 0 &  \none & \none  \\
& &  0 & \none & \none \\
& &\none &\none & \none  \\
\none & & & & \\
\none & & 1 & \none & \none 
\end{ytableau}
\end{array}
\]
Let us look at the first cell. The quotients $M/S$ with $S\subset \soc (M_{(3)})$ can be visualized by the left figure below.
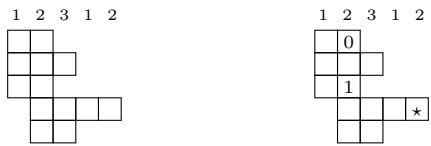
\[ 
\begin{array}{c}
\makebox[9pt][c]{\text{\tiny{1}}}\makebox[9pt][c]{\text{\tiny{2}}}\makebox[9pt][c]{\text{\tiny{3}}}\makebox[9pt][c]{\text{\tiny{1}}}\makebox[9pt][c]{\text{\tiny{2}}}\\
\ytableausetup{smalltableaux}
\begin{ytableau}
{} & &\none  & \none & \none  \\
& & & \none & \none \\
& &\none &\none & \none  \\
\none & & & & \\
\none & & & \none & \none 
\end{ytableau}
\end{array}
\quad \quad \quad \quad \quad \quad  
 \begin{array}{c}
\makebox[9pt][c]{\text{\tiny{1}}}\makebox[9pt][c]{\text{\tiny{2}}}\makebox[9pt][c]{\text{\tiny{3}}}\makebox[9pt][c]{\text{\tiny{1}}}\makebox[9pt][c]{\text{\tiny{2}}}\\
\ytableausetup{smalltableaux}
\begin{ytableau}
{} & 0& \none  & \none & \none  \\
& & & \none & \none \\
& 1 &\none &\none & \none  \\
\none & & & & \star  \\
\none & &  & \none & \none 
\end{ytableau}
\end{array}
\]
The right hand side figure above shows a cell in the socle at $2$ with pivot box in row $3$. 
In the figure above on the left each box corresponds to the residue class of a basis vector of $M$ corresponding to a box at the same place. This gives an ordered basis of $M/S$.
To parametrize the cells in the complete flag variety, we need to parametrize the cells in the socles of the successive quotients. Now, to parameterize the cells in $\FMd$ with $\ddd$ complete, we put an $r=\dim M$ into an end box in the column specified by $\dd^1$. Then, we put an $r-1$ into an end or not filled box in the column specified by $\dd^2-\dd^1$, etc. We obtain a filling of $Y_M$ with numbers 
$1,\ldots , r$ which is increasing from left to right in the rows, we call such a filling a \textbf{row multi-tableau} of dimension $\ddd $. Let $\tau $ be a row multi-tableau, then we write $C_{\tau} \subset \FMd$ for the corresponding cell. We can recover a dimension filtration $\ddd=: \Dim \tau$ and an isomorphism class of a module $Y_{\tau}$ from $\tau$ by counting boxes in the columns and by looking at the shape of $\tau$.   
For $k\in \{1, \ldots , r\}$ we write $c_k\in \Z/n$ for the column containing 
$k$ and $r_k\in \N$ for the row containing $k$. We define  
\[ d_{\tau }(k) := \# \{ s\in \{1, \ldots , k-1\}\mid c_s=c_k, r_s>r_k, r_s\neq r_t, \; s<t<k\}
\]
One has by construction
\[ \dim C_{\tau } =\sum_{k=1}^r d_{\tau}(k) .\]  

For example, for $\tau $ as follows, $d_{\tau}(-)$ is as in the table below and therefore $\dim C_{\tau}=9$.
\[
\tau= 
\begin{array}{c}
\makebox[9pt][c]{\text{\tiny{1}}}\makebox[9pt][c]{\text{\tiny{2}}}\makebox[9pt][c]{\text{\tiny{3}}}\makebox[9pt][c]{\text{\tiny{1}}}\makebox[9pt][c]{\text{\tiny{2}}}\\
\ytableausetup{smalltableaux}
\begin{ytableau}
5 & \elf & \vierzehn &\none & \none  \\
2 & 6 &  8& \none & \none \\
3 & \dreizehn &\none &\none & \none  \\
\none & 1 & 9  &  \zehn & \zwoelf \\
\none & 4 & 7 & \none & \none 
\end{ytableau}
\end{array}
\]

\begin{tabular}[c]{l|c|c|c|c|c|c|c|c|c|c|c|c|c|c}
$k$ & $1$ & $2$ & $3$ & $4$ & $5$ & $6$ & $7$ & $8$ & $9$ & $10$ & $11$ & $12$ & $13$ & $14$ \\
\hline 
$d_{\tau }(k)$ & $0$ & $0$ & $0$ & $0$ & $2$ & $2$ & $0$ & $1$ & $1$ & $0$ & $0$ & $0$ & $1$ & $2$ 
\end{tabular}

\noindent
\hspace{1cm}

Now let $Y$ be a tuple of $n$ finite sequences of positive numbers. 
For the $i$-th sequence we write a skew-diagram starting at column $i\in \Z/n$ with row lengths specified by the sequence of positive numbers. 
Any (complete) numbering of the boxes of $Y$ which is increasing in the length is called a row multi-tableau. From the columns of the 
tableau, we can read of the dimension filtration denoted by $\Dim \tau$, we set $\dim \tau := \sum_{k=1}^r d_{\tau}(k)$ and we write 
$Y_{\tau}$ for the $Y$ underlying $\tau$.

\[
\begin{aligned}
\fYd &:= \#\{ \tau \mid Y_{\tau }=Y, \Dim \tau =\ddd \}\\
\fYdk &:= \#\{\tau \mid  Y_{\tau}=Y, \Dim \tau =\ddd, \dim \tau =k\}
\end{aligned}
\]
and for $M$ of shape $Y$ one has $\dim H^*(\FMd ) =\fYd$. Moreover $H^{odd} (\FMd)=0$ and $\dim H^{2k} (\FMd)=\fYdk$.

We can use the following recursion to 
calculate $\fYd$ and $\fYdk$. For $i\in \Z/n$ we denote by $E(Y,i)=(b_1,\ldots , b_s)$ the tuple of right end boxes in $Y$ in column $i$ ordered from top to bottom. For $b \in E(Y,i)$ we write $Y\setminus b$ for the skew-diagram obtained from $Y$ by removing the box $b$. We write the dimension filtration as the sequence of the supports of the subquotients $\ddd= (i_1, \ldots , i_r) \in (Q_0)^r$, i.e. $\dd^1=e_{i_1}, \dd^2=e_{i_1}+e_{i_2}$, etc. and let $E(Y,i_1)=(b_1, \ldots , b_s)$ 
\[ 
\begin{aligned}
\fYd &=  \sum_{b\in E(Y, i_1)} \fYbd \\
\fYdk &= \sum_{m=1}^s \fYbdm
\end{aligned}
\]

\section{Torus fixed points and GKM-theory}

Let $Q$ be an arbitrary quiver and $\mathbb{A}_{r+1}$ the quiver $0\to 1 \to \cdots \to r$.   
Let $\La =KQ\otimes_{K}K\mathbb{A}_{r +1}$. Then a $\La$-module is given by a sequence of $KQ$-module homomorphisms $U_0\xrightarrow{f_0}U_1\xrightarrow{f_1} \cdots \xrightarrow{f_{r -1}} U_{r}$ and let $\mcM $ be the full subcategory of $\La$-modules with $U_0 =0$ and every morphism $f_k$ is a set-theoretic inclusion.

\begin{defi} An object $0\subset U_1\subset \cdots \subset U_r =M$ in $\mcM $ is called \textbf{split} with respect to a direct sum decomposition $M=M_1\oplus \cdots \oplus M_t$ in indecomposable submodules if 
\[ U_s =(U_s\cap M_1)\oplus \cdots \oplus (U_s\cap M_t), \quad  1\leq s\leq r-1 .\]
We call a $\La$-module $U$ in $\mcM $ \textbf{split} if there exists a direct sum decomposition of $U_r$ such that it is split with respect to this direct sum decomposition. 
Fix an additive family of modules with for each isomorphism type precisely one module. The full subcategory of $\mcM $ of split modules with respect to this family defines an additive subcategory $S\mcM $ of $\mcM $.  
\end{defi}

Let $U:=(0\subset U_1\subset \cdots \subset U_r=M)$ be a split module with respect to $M=M_1\oplus \cdots \oplus M_t$, we set 
$(U\cap M_i):= \bigl(0\subset (U_1\cap M_i) \subset \cdots \subset (U_n \cap M_i) =M_i\bigr) $, then $U=(U\cap M_1)\oplus \cdots \oplus (U\cap M_t)$ is a direct sum decomposition into indecomposable $\La$-submodules. 

Let $M= \bigoplus_{j =1}^t M_i$ be a decomposition into indecomposables.  
\[T:=\{ \bigoplus_{j} \la_{j} \id_{M_j} \mid \la_{j} \in K^*\} \subset  \prod_{j}\Aut (M_j) \subset \Aut (M)\]

\begin{lemma} Let $U\in\FMd$. Then, $U$ is split with respect to the decomposition of $M$ if and only if it is a $T$-fixed point. 
\end{lemma}

\begin{proof} Clearly, if $U$ is split, it is fixed under $T$. 
Conversely, let $U\in (\FMd)^T$. Take $x\in U_s$, we need to see that it can be written as $\sum_j x_j$ with $x_j \in U_s \cap M_j$. Write $x=x_j+x^\prime$ with $x^\prime \in M^\prime = M_1\oplus \cdots \oplus M_{j-1} \oplus M_{j+1} \oplus \cdots \oplus M_t$, $x_j\in M_j$ and we claim $x_j\in U_s $. Let $\la \neq 1$ and apply 
$f=(\la -1)\id_{M_j} \oplus (-1)\id_{M^\prime} \in T$, then by assumption $f(x)\in U_s$ and therefore $x+f(x)=\la x_j\in U_s$ which implies $x_j\in U_s$.  
\end{proof}

We remark that for $M$ indecomposable all points in $\FMd$ are $T$-fixed.  
Assume now $M$ is uniserial. Then quiver flag varieties $\FMd$ are either empty or a point, given by an indecomposable $\Lambda$-module $M[\ddd]$. 

\begin{rem} Let $M=M_1\oplus \cdots \oplus M_t$ be a module such that each $M_i$ is uniserial. For $U\in \FMd$ one has $U$ is a $T$-fixed point if and only if one of the following equivalent conditions is fulfilled   \\
(*) $U=\bigoplus_j M_j[\ddd^j]$ as $\Lambda$-module with $\sum_{j=1}^t \ddd^j=\ddd$. \\
(*) $U$ is the image of a direct sum map $\oplus \colon \FMdone \times \cdots \times \FMdt \to \FMd$ with $\sum_{j=1}^t \ddd^j=\ddd$. \\
In particular, $(\FMd)^T$ is a finite set. 
\end{rem}

From now on we investigate the case of nilpotent representation of the oriented cycle. 
We define the modules $M[\ddd]$ for $M=E_i[\ell]$ as follows. 
Let $\la =(\ell_1\geq \ell_2\geq \cdots \geq \ell_r)$ be a partition with at most $r$-parts, $\ell_1=\ell, \ell_j\geq 0$. We define an element in $\mcM$ as follows 
\[ 
E_i[\la ] := (0\subset E_i[\ell_r]\subset \cdots \subset E_i[\ell_1] )
\]
where $E_i[0]:=0$. The category $S\mcM $ is Krull-Schmidt. 
The isomorphism classes of indecomposable objects in 
$S\mcM$ are given by the $E_i[\la], i\in \Z/n$ and $\la$ a partition. One has $\End_{\La}(E_{i}[\la])=\End_{A} (E_i[\ell_1 ])$.

Now, if $\ddd$ is a complete dimension filtration, then the split modules $U\in \FMd$ correspond to the points in the cells $C_{\tau}$ which are given by the inclusion of the pivot box in each successive step (compare previous section). To each row in a row-multi-tableau $\tau$ we assign a split module as follows. Assume the row ends at $i$ and has a filling with positive numbers $a_1<a_2<\cdots <a_{\ell}$, then take 
$E_i[\la]$ with $\la = (0^{m_0}, 1^{m_1} , \ldots , \ell^{m_{\ell}}) $ where $\sum_j m_j =r$ such that $m_j$ is the number of occurences of $j$ given as follows 
\[m_0=r-a_{\ell}, m_1=a_{\ell}-a_{\ell-1}, m_2=a_{\ell -1}- a_{\ell-2}, \ldots , m_{\ell-1} = a_{2}-a_1, m_{\ell}=a_1. \] 

\begin{lemma} Assume $\ddd$ is a complete dimension filtration. 
Then, every cell $C_{\tau}$ contains a unique $T$-fixed point. 
In particular, there is a bijection between isomorphism classes of split modules and row root tableau. 
\end{lemma}

It is easy to see that the cells $C_{\tau}$ are $T$-equivariant. In fact, our cell decomposition coincides for $K=\C$ with a Bialynicki-Birula cell decomposition, see \cite{Car}. 

If we choose in the Spaltenstein fibration the split inclusions as reference points $U_0$ each factor in the splitting is $T$-invariant. We conclude that there are $\dim C_{\tau}$ one-dimensional $T$-orbits in $C_{\tau}$. 

Each one-dimensional $T$-orbit closure contains precisely two $T$-fixed points, the one corresponding to the split module in $C_{\tau}$ and the other 
one obtained by an \textbf{admissible} swap of box entries from $\tau$, defined as follows: Another row multi-tableau $\tau^\prime$ is obtained by an admissible swap 
from $\tau$ if there is a connected interval of box-entries in two rows let's say $r_k$ and $r_m$ (where $k$ and $m$ are the largest entries in these intervals) which are swapped to obtain $\tau^\prime$ and $\left|\dim C_{\tau}- \dim C_{\tau^\prime}\right| =1$. In this case we write $\tau^\prime := \si_{k,m}(\tau)$. 

From the $T$-fixed points $F:=\{\tau\mid C_{\tau}\subset \FMd\}$ and the one-dimensional $T$-orbit closures we obtain a graph $\Gamma$ which determines a ring 
 \[
 S(\Gamma ) := \left\lbrace (f_{\tau})_{\tau \in F} \in \C[x_1, \ldots , x_t]^{\oplus F}\mid 
\begin{aligned} 
&  \\
  &f_{\tau} - f_{\si_{k,m}(\tau )} \in (x_k-x_m) \C[x_1, \ldots , x_t]\\
  &\text{ for all admissible swaps}
  \end{aligned}
  \right\rbrace
 \] 
with $t=\rk T$.

Let us denote the inclusion of the torus fixed points by $i\colon (\FMd)^T \to \FMd$. The $T$-equivariant cohomology with complex coefficients of $\FMd$ can be described directly by applying GKM-theory, see \cite[Thm 1.2.2]{GKM}. 
\begin{coro}  Let $K=\C$ and $\ddd$ a complete dimension filtration. Then, the map $i^*$ induces an injective ring homomorphism 
\[i^*\colon H_T^*(\FMd) \to H_T^*((\FMd)^T)\]
with image isomorphic to $S(\Gamma )$ described before.  
\end{coro}

\section{Kato's standard modules}
Let $K=\C$, $\dd\in \N_0^{Q_0}$, $G:= \prod_{i\in Q_0} \Gl_{d_i}(K)$ and one has the $G$-representation $\Rd:= \bigoplus_{i\to i+1} \Hom (K^{d_i} , K^{d_{i+1}})$ given by 
$(g_i)\cdot (m_{i\to i+1}) := (g_{i+1} m_{i\to i+1} g_i^{-1})$. We set 
$I_{\dd}:= \{ \ddd:=(i_1, \ldots , i_r )\mid \sum_{j=1}^r i_j =\dd \}$. 
For each $\ddd \in I_{\dd}$ we fix a $Q_0$-graded flag $(V^{k})_{0\leq k \leq r}$ in 
$\bigoplus_{i\in Q_0}K^{d_i}$ of dimension vector $\Dim V^k/V^{k-1}= i_k$. We define a parabolic group $P(\ddd )\subset G$ as the stabilizer of $V^\bullet$ and a $P(\ddd )$-subrepresentation of $\Rd$ by $F(\ddd ):=\{ (m_{i\to i+1}) \in \Rd \mid m_{i\to i+1} (V^k_i)\subset V_{i+1}^{k-1}\}$. 
We set $\pi_{\ddd} \colon E_{\ddd }:= G\times^{P(\ddd )} F(\ddd )\to \Rd, \quad \overline{(g, f)} \mapsto gf$ for the collapsing of the associated vector bundle
$E_{\ddd}$ over $G/P(\ddd )$ and $e_{\ddd} :=\dim_{\C}E_{\ddd}$. Then we define 
\[ \mcL:= \bigoplus_{\ddd \in I_{\dd}} (\pi_{\ddd })_* \underline{\C}[e_{\ddd}] \quad \in D^b_G(\Rd)
\]
where $D^b_G(\Rd)$ is the equivariant derived category of Bernstein and Lunts \cite{BL}. 
Let $\La$ be the set of multipartitions $\la=(\la^1, \ldots , \la^n)$ corresponding to the $G$-orbits $\mcO_{\la}$ in $\Rd$. We fix a point $i_\la\colon \{M\} \inj \mcO_{\la}$.    
The module corresponding in our situation to Kato's standard module $K_{\la}$ shifted by $(-\dim \mcO_{\la})$ is given by 
\[ 
K_{\la}\langle -\dim \mcO_{\la} \rangle := H^*(i_{\la}^{!} \mcL ) = \bigoplus_{\ddd\in I_{\dd}} H_{e_{\ddd} -*}(\FMd )
\]
where $H_*()$ refers to Borel-Moore homology with complex coefficients. 
By the local to global sequence in Borel-Moore homology we get $\dim H_k(\FMd) = \fldk$.  
For any $\Z$-graded vector space $V$ with $V_j=0$ for $j <<0$ we define 
the graded dimension to by 
$ \gdim V := \sum_{j\in \Z} (\dim V_j) t^j \; \; \in (\Z\left[\left[ t \right]\right])\left[t^{-1}\right].  $
Using our cell decomposition we have  
\[
t^{-\dim \mcO_{\la}} \gdim K_{\la} = \sum_{k\in \Z} \left( \sum_{\ddd \in I_{\dd}} \fldek \right) t^k. 
\] 
\begin{rem}
Consider the BBD-decomposition of $\mcL$, apply  $H^*(i_{\la}^!(-))$ and then apply $\gdim$, the resulting equation in the ring of Laurent series expresses\\ $t^{-\dim \mcO_{\la}}\gdim K_{\la}$ as a sum over products given by $\gdim L_{\mu}$ with 
 $\gdim \mcH^*_{\mcO_{\la}}(IC_{\mu})$ where $L_{\mu}$ runs through the simple modules (over the quiver Hecke algebra of the nilpotent $\dd$-dimensional representations of the quiver $Q$) and $IC_{\mu}$ is the simple perverse sheaf on $\Rd$ build from the trivial local system on $\mcO_{\mu}$. It is an observation of Schiffmann (see \cite[Corollary 6.2]{Sch}, generalizing from Lusztig for $n=1$) that 
$\gdim \mcH^*_{\mcO_{\la}}(IC_{\mu})$ is an affine Kazhdan Lusztig polynomial of type $A$. It naturally generalizes the notion of a Kostka polynomial in the sense that it describes the base change expressing a canonical basis $[K_{\la}]$ in a PBW-basis $[L_{\mu}]$ within the Grothendieck group of 
finite-dimensional graded modules, see loc. cit.  
\end{rem}

\section*{Acknowledgements}

I would like to thank Peter McNamara for pointing out that the multiplicity spaces can be zero and therefore, Kato's work can not directly be applied. Furthermore, I would like to mention financial support from the CRC 701 in Bielefeld and the SPP 1388 (Schwerpunktprogramm  Darstellungstheorie) by covering travel costs and research stays. 

%%%%%%%%%%%%%%%%%%%%%%%%%%%%%%%%%%%%%%%%%%%%%%%%%%%%%%%%%%%%%%%%%%%%%%%%%%%
%%%%%%%%%%%%%%%%%%%%%%%%%%%%%%%%%%%%%%%%%%%%%%%%%%%%%%%%%%%%%%%%%%%%%%%%%%%%%%
%%%%%%%%%%%%%%%%%%%%%%%%%%%%%%%%%%%%%%%%%%%%%%%%%%%%%%%%%%%%%%%%%%%%%%%%%%% 

%\addcontentsline{toc}{section}{\textbf{References} \hfill}

\bibliographystyle{alphadin}
\bibliography{CellsInQuiverFlags}
\end{document}